\documentclass[preprint]{elsarticle}

\usepackage{amsfonts}
\usepackage{amsmath}
\usepackage{amssymb}
\usepackage[english]{babel}
\usepackage{amsthm}
\usepackage{graphicx}

\newtheorem {theorem}{Theorem}
\newtheorem {lemma}[theorem]{Lemma}
\newtheorem {conj}[theorem]{Conjecture}
\newtheorem {cor}[theorem]{Corollary}
\newtheorem {prop}[theorem]{Proposition}

\theoremstyle{remark}
\newtheorem* {remark}{Remark}

\DeclareMathOperator{\aff}{\it aff}

\DeclareMathOperator{\lin}{\it lin}
\DeclareMathOperator{\fan}{\it Fan}
\DeclareMathOperator{\conv}{\it conv}

\title{An upper bound for a valence of a face in a parallelohedral tiling}
\author[add1,add2]{Alexander Magazinov \fnref{fn1}}
\ead{magazinov-al@yandex.ru}
\address[add1]{Steklov Mathematical Institute of the Russian Academy of Sciences, 8 Gubkina street, Moscow 119991, Russia}
\address[add2]{Yaroslavl State University, 14 Sovetskaya street, Yaroslavl 150000, Russia}
\fntext[fn1]{Supported by the Russian government project 11.G34.31.0053 and RFBR grant 11-01-00633.}

\begin{document}

\begin{abstract}

Consider a face-to-face parallelohedral tiling of $\mathbb R^d$ and a $(d-k)$-dimensional face $F$ of the tiling. We prove that the valence of  
$F$ (i.e. the number of tiles containing $F$ as a face) is not greater than $2^k$. If the tiling is affinely equivalent to a Voronoi tiling for
some lattice (the so called Voronoi case), this gives a well-known upper bound for the number of vertices of a Delaunay $k$-cell. Yet we emphasize that such 
an affine equivalence is not assumed in the proof.

\end{abstract}


\maketitle

\section{Introduction}

The central point of the parallelohedra theory is the famous Voronoi conjecture.

\begin{conj}[Voronoi]\label{voronoi}

Every $d$-dimensional parallelohedron $P$ is affinely equivalent to a Dirichlet-Voronoi domain for some $d$-dimensional lattice.

\end{conj}

Although the conjecture was posed in 1909 in the paper \cite{vor}, it has not been proved or disproved so far in the general case. However,
several significant partial solutions were obtained \cite{del,erd,ord,vor,zhi}.

Let $\mathcal T(P)$ be a face-to-face tiling of $\mathbb R^d$ by parallel copies of a parallelohedron $P$. Choose an arbitrary $(d-k)$-dimensional
face $F$ of the tiling. 

Denote by $\pi$ the orthogonal projection along $\lin F$ onto the complementary affine space
$(\lin F)^\bot$. Then there exists a complete $k$-dimensional cone fan $\fan (F)$ 
({\it the fan of $F$}) that splits $(\lin F)^\bot$ into convex polyhedral cones with vertex $\pi(F)$,
and a neighborhood $U = U(\pi(F))$ such that every face $F'\supset F$ corresponds to a cone $C\in \fan(F)$ satisfying 
$$\pi(F')\cap U = C\cap U.$$
Speaking informally, $\fan(F)$ has the same combinatorics as the transversal section of $\mathcal T(P)$ in a small neighborhood of $F$.

The definition above is equivalent to the definition of a {\it star} of a face $F$, introduced by Ryshkov and Rybnikov \cite{rry}, yet seems to be
more formal.

Studying the combinatorial structure of such cone fans proved to be an effective tool to verify the Voronoi conjecture in special cases \cite{del,ord,zhi}.

For fixed $k$ and varying $d$, it is a complicated problem to classify all possible combinatorial types of $\fan(F)$ if a positive answer to conjecture~\ref{voronoi} is not preassumed. The classification for $k = 2$ has been obtained in \cite{min} and consists of 2 combinatorial
types of cone fans shown in Figure \ref{f1}.

\begin{figure}[h]\label{f1}

      \centerline{\includegraphics[scale=0.75]{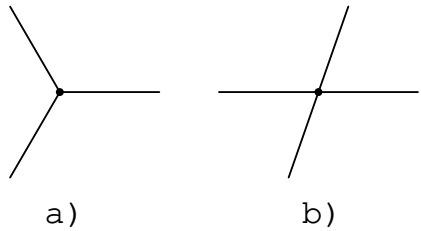}}
      \caption{2 possible fans of $(d-2)$-faces}
  
\end  {figure}%

for $k = 3$ Delaunay \cite{del} proved that every fan of a $(d-3)$-face of a parallelohedral tiling belongs to one of the 5 types shown in Figure \ref{f2}.

\begin{figure}[h]\label{f2}
   
      \centerline{\includegraphics[width = \textwidth]{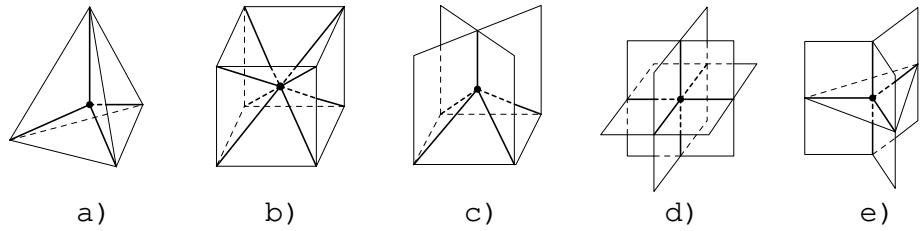}}
      \caption{5 possible fans of $(d-3)$-faces}
   
\end  {figure}

The case $k = 4$ has been partially considered in \cite{ord}, but the complete classification was not obtained, and for $k>4$ almost nothing is known. 
The main goal of this paper is to show that a classification is possible {\it in principle} for every $k$, i.e. the list of all combinatoral types of 
$\fan F$, where $F$ is a $(d-k)$-face, is finite. The idea is to obtain an upper bound for the number of tiling parallelohedra incident to a face $F$.

Let 
$$\nu(F) = card\,\{  P'\in \mathcal T(P) : F\subset P'  \}.$$
$\nu(F)$ will be called the {\it valence} of the face $F$.

\begin{theorem}\label{2k}

Let $F$ be a $(d-k)$-dimensional face of a parallelohedral tiling of $\mathbb R^d$. Then
$$\nu(F)\leq 2^k.$$

\end{theorem}

\begin{remark}

The upper bound $2^k$ is sharp for all integer $d, k$ satisfying $0<k\leq d$; for example, for every $(d-k)$-face $F$ of a cubic tiling of $\mathbb R^d$ holds
$$\nu(F) = 2^k.$$

\end{remark}

Theorem \ref{2k} immediately implies

\begin{cor}\label{finiteness}

Given $k\in \mathbb N$, there exists a set of complete $k$-dimensional cone fans
$$\{\mathcal C^k_1, \mathcal C^k_2, \ldots, \mathcal C^k_{N(k)} \}$$
such that for every $d$, every $d$-parallelohedron $P$ and every $(d-k)$-face $F$ of $\mathcal T(P)$ the fan of $F$ is isomorphic to some $\mathcal C^k_i$.

\end{cor}

\begin{proof}

Obviously, there is only a finite number of combinatorial types of complete $k$-dimensional cone fans splitting $\mathbb R^k$ into no more than 
$2^k$ full-dimensional convex polyhedral cones.
According to theorem \ref{2k}, the fan of $F$ necessarily belongs to one of those combinatorial types.

\end{proof}

For a centrally symmetric polytope $Q$ denote by $c(Q)$ its center of symmetry.

To proceed with the proof of theorem \ref{2k} recall several basic properties of parallelohedral tilings.

\begin{enumerate}

\item A parallelohedron $P$ is a centrally symmetric polytope (see~\cite{min}).

\item The set
$$\Lambda = \{c(P'): P'\in \mathcal T(P)\}$$ 
is a lattice (see also~\cite{min}). Under assumption $\mathbf 0 \in \Lambda$, one can also treat $\Lambda$ as a translation group.

\item \label{c3} If $P_1, P_2 \in \mathcal T(P)$ and $P_1\cap P_2 \neq \varnothing$, then $P_1\cap P_2$ is a centrally symmetric face of $\mathcal T(P)$.
Moreover,
$$c(P_1\cap P_2) = \frac {c(P_1)+c(P_2)}{2}.$$ 
A face $F$ of $\mathcal T(P)$ representable in the form $F = P_1\cap P_2$, where $P_1, P_2 \in \mathcal T(P)$, is called {\it standard} (see~\cite{dol}).

\end{enumerate}

\section{The Voronoi case}

For a better understanding of the aim, first restrict oneself to the Voronoi case. Start considering this case with a folklore lemma. 

\begin{lemma}\label{parity}
Let $P_1, P_2 \in \mathcal T(P)$ be the 2 distinct parallelohedra such that $c(P_1)$ and $c(P_2)$ belong to the same class modulo $2\Lambda$.
(In other words, $P_1$ and $P_2$ belong to the same parity class.) Then $P_1\cap P_2 = \varnothing$.
\end{lemma}

\begin{proof}
According to property 3 of a parallelohedral tiling (see~page~\pageref{c3}),
\begin{equation}\label{50}
\frac{c(P_1) + c(P_2)}{2} \in (P_1\cap P_2).
\end{equation}

On the other hand, since $c(P_1)$ and $c(P_2)$ belong to the same class $\mod{2\Lambda}$,
$$\frac{c(P_1) + c(P_2)}{2} \in \Lambda.$$ 
Therefore $\frac{c(P_1) + c(P_2)}{2}$ is a center of some parallelohedron $P_3$ different from $P_1$ and $P_2$. Thus
$$\frac{c(P_1) + c(P_2)}{2} \in int\, P_3,$$
which is a contradiction to (\ref{50}). Lemma is proved.
\end{proof}

The technique used in the proof is rather standard. For example, similar methods were
used in~\cite{dsa}. (See also \cite[part~1, p.~277]{vor}, the description of parallelohedra of a given parity class adjoint to a given parallelohedron
by a hyperface.)

Let $\mathcal T_V(\Lambda)$ be the Voronoi tiling for some $d$-dimensional lattice $\Lambda \subset \mathbb R^d$. Then $\mathcal T_V(\Lambda)$
has a dual {\it Delaunay} tiling $\mathcal D(\Lambda)$. Every $(d-k)$-face $F$ in $\mathcal T_V(\Lambda)$ has a 
$k$-dimensional dual face $D(F)$ in $\mathcal D(\Lambda)$ such that
\begin{equation}\label{51}
D(F) = \conv \{ c(P): P\in \mathcal T_V(\Lambda)\; \text{and} \; F\in P \}.
\end{equation}
Moreover, the combinatorial structure of $D(F)$ completely describes the combinatorial structure of $\fan(F)$. In particular,
$\nu(F)$ is equal to the number of vertices of $D(F)$.

\begin{prop}\label{vcase}
The statement of theorem~\ref{2k} holds for Voronoi tilings.
\end{prop}

\begin{proof}
Following the argument above, it is sufficient to prove that $D(F)$ has at most $2^k$ vertices. According to the properties of a
lattice Delaunay tiling,
$$dim\, \aff D(F) = k,$$
so the set of vertices of $D(F)$ is a subset of some $k$-dimensional lattice $\Lambda(F) \subset \Lambda$.

By lemma~\ref{parity}, every 2 vertices of $D(F)$ belong to different classes $\mod{2\Lambda}$, otherwise 2 parallelohedra of the same parity class
had to intersect at $F$. Therefore every 2 vertices of $D(F)$ belong to different classes $\mod{2\Lambda(F)}$. Since $\Lambda(F)$
has exactly $2^k$ classes $\mod{2\Lambda(F)}$, there are at most $2^k$ vertices of $D(F)$. Proposition is proved.
\end{proof}

\begin{remark}
Another approach to estimate the number of vertices of $D(F)$ is described in~\cite[Proposition~13.2.8]{dla}.
\end{remark}

Thus, if the Voronoi conjecture (conjecture~\ref{voronoi}) has a positive answer, theorem~\ref{2k} is proved.

In the way similar to corollary~\ref{finiteness}, proposition~\ref{vcase} implies that there are finitely many combinatorial types of $D(F)$ for fixed 
$k$ and regardless of $d$. Moreover, an algorithm for classification all possible lattice Delaunay $k$-cells is given in \cite{dsv} An idea of such
algorithms certainly belongs to Voronoi \cite{vor} who constructed an algorithm to classify all possible combinatorial types of Voronoi parallelohedra.

For an arbitrary parallelohedral tiling $\mathcal T(P)$ and its $(d-k)$-face $F$ it is also possible to introduce in the similar way as (\ref{51}) the set
$$ D(F^{d-k}) = \conv \{c(P') : P'\in \mathcal T(P) \; \text{and} \; F^{d-k}\subset P' \}. $$
As far as author knows, there are no satisfactory results on $\dim\aff  D(F)$ in the general case.

In the case 
\begin{equation}\label{10}
dim\,\aff  D(F) \leq k
\end{equation}
theorem~\ref{2k} easily follows from lemma~\ref{parity}. Yet the inequality (\ref{10}) remains an open problem.

\section{Outline of the proof}

The methods described above essentially involve the assumption that conjecture~\ref{voronoi} has a positive answer. However, theorem~\ref{2k}
can be proved in non-Voronoi case as well, by exploiting several other ideas.

The proof of theorem \ref{2k} consists of the 3 main steps.

\begin{enumerate}

\item Construct a set 
$$\{ F_1 = F, F_2, \ldots, F_m \}$$ 
of all faces of a parallelohedron $P_0 \in \mathcal T(P)$ such that every 2 faces of the set are equivalent by a $\Lambda$-translation. Prove
that $\nu(F_1) = m$.

\item Refine the notion of {\it an antipodal set} given in~\cite{dgr}. Prove that the set
$$W = \{\pi(F_i) : i = 1, 2, \ldots, m\},$$
is antipodal (here $\pi$ is the above defined projection along $\lin F$ onto the complementary space $(\lin F)^\bot$).

\item Estimate the cardinality of an arbitrary antipodal set in $\mathbb R^k$.

\end{enumerate}

The third step uses the technique introduced by Danzer and Gr\"unbaum in \cite{dgr}. However, \cite{dgr} deals with antipodal full-dimensional
point sets in $\mathbb R^k$, i.e. the sets $W$ satisfying $dim \aff W = k$. Although Danzer and Gr\"unbaum's theorem cannot be used
directly here, it is not hard to extend the technique to the class of antipodal sets satisfying the refined definition.

\section{Faces equivalent by translation}

To introduce a uniform notation, put $F_1 = F$. Choose a parallelohedron $P_0 \in \mathcal T(P)$ such that $F_1 \subset P_0$. Let
$$\{F_1, F_2, \ldots, F_m\}$$
be the set of all faces of $P_0$ equivalent to $F_1$ up to a $\Lambda$-translation.

Denote by $\mathbf t_{ij}$ the vector of $\Lambda$-translation such that
$$F_i + \mathbf t_{ij} = F_j.$$
For every $i,j \in \{1, 2, \ldots, m\}$ define
$$P_{ij} = P_0 + \mathbf t_{ij}.$$
Clearly, $P_{ij}\in \mathcal T$.

\begin{lemma}\label{valence}

$\nu(F_1) = m$.

\end{lemma}

\begin{proof}

Since $F_i \subset P_0$, one has
$$F_1 = F_i + \mathbf t_{i1} \subset P_0 + \mathbf t_{i1} = P_{i1}.$$

Therefore there are at least $m$ parallelohedra of $\mathcal T(P)$ meeting at $F_1$, because the parallelohedra
$$P_0 = P_{11}, P_{21}, P_{31}, \ldots, P_{m1}$$
are pairwise different.

Suppose there exists one more parallelohedron $P'\in \mathcal T(P)$ such that $F_1\subset P'$. Thus there is a non-zero $\Lambda$-translation $\mathbf t$
such that $P' = P_0 + \mathbf t$. Obviously, $\mathbf t \neq \mathbf t_{i1}$, so $-\mathbf t\neq \mathbf t_{1i}$ for
every $i=1,2,\ldots, m$.

Let $F' = F_1 - \mathbf t$. Since $F_1 \subset P'$,
$$F' = F_1 - \mathbf t \subset P' - \mathbf t = P_0.$$
Consequently, $F'$ is a face of $P_0$ equivalent to $F_1$ up to a $\Lambda$-translation, hence $F' = F_i$ for some $i\in \{2,3,\ldots, m\}$.
By definition of $\mathbf t_{1i}$, one has 
$$F_1 - \mathbf t = F' = F_i = F_1 + \mathbf t_{1i}.$$
Therefore $-\mathbf t = \mathbf t_{1i}$, which is a contradiction. So, there are exactly $m$ parallelohedra of $\mathcal T(P)$ meeting at $F_1$, namely
$$P_0, P_{21}, P_{31}, \ldots, P_{m1}.$$
Thus $\nu(F_0) = m$.

\end{proof}

\section{Constructing an antipodal set}

Call a finite set $W\subset \mathbb R^k$ {\it antipodal} if for every pair of distinct points $x, y \in W$ there exists a pair of {\it distinct}
parallel hyperplanes $\beta, \gamma$ such that
$$x\in \beta, \quad y\in \gamma,$$
and $W$ lies between $\beta$ and $\gamma$. 

Recall that $\pi$ is a projection along $\lin F_1$ onto the complementary space $\mathbb R^k$.

\begin{lemma}\label{antipod}

Let $w_i = \pi(F_i)$. Then 
$$W = \{ w_i: i = 1, 2, \ldots, m \}$$
is an antipodal set in $\mathbb R^k$.

\end{lemma}

\begin{proof}

It is sufficient to show that for every integer $i,j, 1\leq i<j\leq m$ there exist two distinct parallel hyperplanes
$\gamma_{ij}$ and $\gamma_{ji}$ such that
$$w_i \in \gamma_{ij}, \quad w_j\in \gamma_{ji},$$
and $W$ lies between these hyperplanes.

In $\mathbb R^d$ take the hyperplane $\Gamma_{ij}$ that separates the parallelohedra $P_0$ and $P_{ji}$ from each other. ($\Gamma_{ij}$ has to be a 
supporting hyperplane to each of these 2 parallelohedra.)
Since $F_j\subset P_0$,
$$F_i = F_j + \mathbf t_{ji} \subset P_0 + \mathbf t_{ji} = P_{ji},$$
and therefore $F_i \subset P_0\cap P_{ji} \subset \Gamma_{ij}$.

By definition, put $\Gamma_{ji} = \Gamma_{ij} - \mathbf t_{ji}$. The hyperplane $\Gamma_{ij}$ is supporting to $P_{ji}$, so the
hyperplane $\Gamma_{ij}$ is supporting to $P_{ji} - \mathbf t_{ji} = P_0$. Also, since $F_i \subset \Gamma_{ij}$,
$$F_j = F_i - \mathbf t_{ji} \subset \Gamma_{ij} - \mathbf t_{ji} = \Gamma_{ji}.$$

Thus $\Gamma_{ij}$ and $\Gamma_{ji}$ are two parallel supporting hyperplanes satisfying
$$F_i \subset \Gamma_{ij}, \quad F_j\subset \Gamma_{ji}.$$

Now define 
$$\gamma_{ij} = \pi (\Gamma_{ij}) \quad \text{and} \quad \gamma_{ji} = \pi (\Gamma_{ji}).$$
Since $P_0$ lies between $\Gamma_{ij}$ and $\Gamma_{ji}$, then $W$ lies between $\gamma_{ij}$ and $\gamma_{ji}$. The hyperplanes $\gamma_{ij}$ and $\gamma_{ji}$
are distinct because the point $\pi(c(P_0))$ lies strictly between them. Thus the required hyperplanes are constructed for every pair
of points $(w_i, w_j)$. Lemma \ref{antipod} is proved.

\end{proof}

\section{Cardinality of an antipodal set}

\begin{lemma}\label{card}

Let $W = \{w_1, w_2, \ldots, w_m\} \subset \mathbb R^k$ be an antipodal set. Then $m\leq 2^k$.

\end{lemma}

\begin{proof}

Denote by $H_x^a$ the homothety with center $x$ and coefficient $a$.

Take an arbitrary $a\in (0, \frac 12)$ and prove that
$$ H_{w_i}^a(\conv  W) \cap H_{w_j}^a(\conv  W) = \varnothing. $$

Indeed, let $\beta_{ij}$ be the hyperplane parallel and equidistant to the hyperplanes $\gamma_{ij}$ and $\gamma_{ji}$ . Then 
the sets $H_{w_i}^a(\conv  W)$ and $H_{w_j}^a(\conv  W)$ lie in different open half-spaces in respect to $\beta_{ij}$ and
therefore do not intersect.

Let $k'\leq k$ be the affine dimension of $W$. Since 
$$H_{w_i}^a(\conv  W)\subset \conv W$$
and because the sets $H_{w_i}^a(W)$ are pairwise non-intersecting, one has
\begin{equation}\label{60}
vol_{k'}(\conv  W)\geq \sum\limits_{i=1}^m vol_{k'}(H_{w_i}^a(\conv  W)) = m\cdot a^{k'}\,vol_{k'}(\conv  W),
\end{equation}
where $vol_{k'}$ stands for the $k'$-dimensional volume.

Further, because $vol_{k'}(W)>0$, (\ref{60}) implies 
\begin{equation}\label{61}
m \leq a^{-k'}.
\end{equation}
The inequality (\ref{61}) holds for every $0<a<\frac 12$, so it holds also for $a = \frac 12$:
$$m  \leq 2^{k'}\leq 2^k,$$

Lemma \ref{card} is proved.

\end{proof}

Now the statement of theorem \ref{2k} is easily obtained by combining lemma~\ref{antipod} and lemma~\ref{card}.  \newline\qed

\section*{Acknowledgements}

There are a number of people who participated in the discussion of ideas and results, namely, R.~Erdahl, F.~Vallentin, A.~Sch\"urmann, A.~Garber, 
A.~Gavrilyuk and M.~Kozachok. 

The research was partially done at Fields Institute, Toronto, Canada during the Thematic Semester of Discrete Geometry and Applications and
at Queen's University, Kingston, Canada with an invitation of prof R.~Erdahl. Author also appreciates the effort by the Laboratory of Geometrical 
Methods in Mathematical Physics at MSU, Moscow to make possible the visit to Canada.

Finally, author acknowledges prof. N.~Dolbilin for scientific guidance, an introduction to the theory of parallelohedra and numerous
useful comments on the text.

\end{document}